\documentclass[12pt]{amsart}
\usepackage{amsfonts}
\usepackage{amsmath}
\usepackage{amssymb}
\usepackage{hyperref}
\usepackage{caption}
\usepackage[margin=1.2in]{geometry}
\usepackage{xcolor}
\usepackage{graphicx}

\newtheorem{theorem}{Theorem}[section]
\newtheorem{proposition}[theorem]{Proposition}
\newtheorem{corollary}[theorem]{Corollary}
\newtheorem{lemma}[theorem]{Lemma}

\theoremstyle{definition}
\newtheorem{definition}[theorem]{Definition}
\newtheorem{example}[theorem]{Example}
\newtheorem{remark}[theorem]{Remark}

\makeatletter
\@namedef{subjclassname@2020}{
  \textup{2020} Mathematics Subject Classification}
\makeatother

\numberwithin{equation}{section}
\numberwithin{theorem}{section}

\numberwithin{equation}{section}

\begin{document}

\title{Distributions in spaces with thick submanifolds}

\author[J. Ding]{Jiajia Ding }
\address{J. Ding, School of Mathematics\\
Hefei University of Technology\\
Hefei 230009\\ China}
\email{dingjiajia@hfut.edu.cn}

\author[J. Vindas]{Jasson Vindas}
\address{J. Vindas\\ Department of Mathematics: Analysis, Logic and Discrete Mathematics\\ Ghent University\\ Krijgslaan 281\\ 9000 Ghent\\ Belgium}
\email{jasson.vindas@UGent.be}

\author[Y. Yang]{Yunyun Yang}
\address{Y. Yang, School of Mathematics\\
Hefei University of Technology\\
Hefei 230009\\ China}
\email{yangyunyun@hfut.edu.cn}

\thanks{This work was supported by the National Natural Science Foundation of China grant number 12001150,  the Hefei University of Technology grant number 407-0371000086, the Ghent University grant number bof/baf/4y/2024/01/155, and the Research Foundation–Flanders grant number
G067621N}


\begin{abstract}
We present the construction of a theory of distributions (generalized functions) with a ``thick submanifold'', that is, a new theory of thick distributions on $\mathbb{R}^n$ whose domain contains a smooth submanifold on which the test functions may be singular. We define several operations, including ``thick partial derivatives'', and clarify their connection with their classical counterparts in Schwartz distribution theory. We also introduce and study a number of special thick distributions, including new thick delta functions, or more generally thick multilayer distributions along a submanifold.
\end{abstract}

\maketitle

\section{Introduction}

In order to solve some problems appearing in applying Schwartz distributions to certain questions in physics and engineering, Fulling, Estrada, and Yang \cite{EF2007,YE2013} constructed a theory of ``thick distributions''. This theory generalizes the classical Schwartz distribution theory: it allows for a point singularity in the domain of the test functions (and hence of the distributions). By doing so, they were able to rigorously define elements such as the ``one-sided Dirac delta function'' and its multidimensional generalizations. The new theory of thick distributions has found more and more applications in physics and partial differential equations (PDE). Thus it  has started to attract attention from researchers in different areas. Let us list a few recent applications here.

Reference \cite{antontsev2024strong} investigated the multidimensional initial-boundary value problem for the semilinear pseudoparabolic equation with a regular nonlinear minor term, which, in general, may be superlinear.   It studied the class of pseudoparabolic PDE of the form $\partial_{t}u=\Delta_{x}u+\Delta_{x}\partial_{t}u+f\left(  x,t,u\right).$ 
When $f\left(  x,t,u\right)  =-\varphi_{n}\left(  t\right)  \left\vert u\right\vert
^{q\left(  x\right)  -2}u $ they discussed regular weak and strong solutions in a multidimensional bounded domain $\Omega$
while allowing the source $f(x,t,u)$ to be superlinear with respect to $u$.
The theory of thick distributions was used to generalize the theory of impulsive partial differential and integro-differential equations.

The Hamiltonian form provides a powerful tool for calculating approximate analytical solutions to the Einstein field equations. In \cite{schafer2024hamiltonian}, compact binaries with spin components are treated by the tetrad representation of general relativity. When treating the configuration, the absolute value of the spin vector can be regarded as a constant, and the compact object is modeled by the Dirac delta function and its derivatives. 

Solutions to Maxwell's equations for arbitrary charge and current density cannot usually be derived analytically, and numerical solutions require long computational time and large amounts of memory. Thus, semi-analytical closed-form solutions (e.g., solutions involving series expansions) are sometimes needed. Reference \cite{le2024time} systematically derives time-domain semi-analytical expressions for the electromagnetic field radiated by a time-varying localized source using knowledge of thick distributions. This time-domain method is valid for any frequency and is independent of the source size.

Moreover, in \cite{Yang2022}, Yang further developed the theory of thick distributions, allowing test functions to have a singular curve in thier domain in $\mathbb{R}^3$ instead of just a singular point. Using this new tool, the author was able to clarify the information in every direction of each cross section of a very thing blood capillary, versus the previous model where the reaction and diffusion on each cross-section are identical. Thus a more refined model was proposed to modify the previous model \cite{vakoc2009three,zheng2013continuous} of reaction and diffusion of growth factors  of a very thin blood capillary in a bulk tumor. In the same article a solution of the corresponding PDE was also given using the Fourier transform of thick distributions, a theory that was developed by Estrada, Vindas, and Yang in \cite{EVY}.

This article is a generalization of \cite{Yang2022}. In this work, we consider a general ``singular submanifold'' instead of just a singular curve. The first few sections rely on some knowledge of the theory of differential manifolds; in particular, the tubular neighborhood theorem plays a central role in our constructions. In Section \ref{space test functions}, we introduce the test function spaces   $\mathcal{D}_{\ast, \Sigma}(\Omega)$ and $\mathcal{E}_{\ast, \Sigma}(\Omega)$, which are topological vector spaces of smooth functions on the open subset $\Omega\setminus \Sigma \subset \mathbb{R}^{n}$, with $\Sigma$ a closed smooth submanifold. Their elements are in general singular along $\Sigma$, but posses suitable asymptotic expansions with respect to tubular coordinates when approaching the submanifold $\Sigma$. 

The asymptotic properties of our test functions on a tubular neighborhood of $\Sigma$ are the key to defining  the corresponding spaces of thick distributions along the submanifold $\Sigma$ by duality in Section \ref{thick distributions section}.  A crucial feature of $\mathcal{D}_{\ast, \Sigma}(\Omega)$ that we establish in Section \ref{space test functions} is its invariance under partial derivative operators.
The latter allows us to define a notion of partial derivatives for $\Sigma$-thick distributions. Several other operations for our new spaces of $\Sigma$-thick distributions as well as their connection with classical Schwartz distribution theory are discussed in Section  \ref{thick distributions section}. In Section \ref{Section examples}, we give a few important concrete examples of thick distributions along a submanifold and show how to compute their partial derivatives. These instances of thick distributions include new thick delta functions, or more generally, thick multilayer distributions along $\Sigma$.

The applications we have mentioned above give an indication of the potential use of our thick distribution theory for some problems in physics and PDE theory, supplying now a new framework for treating singularities along arbitrary smooth submanifolds. Furthermore, several operations that are not well-defined for Schwartz distributions and often require cumbersome treatments in the literature, such as the radial \cite{BSV2017} or normal-type derivatives just to mention a few, can in contrast be naturally defined for thick distributions, which makes thick distribution theory particularly useful. 
\section{Preliminaries}
 Throughout this article, we fix a (boundaryless) closed submanifold $\Sigma \subset \mathbb{R}^{n}$ of \emph{codimension} $d$. We point out that we do not assume that $\Sigma$ is orientable (unless otherwise explicitly stated like in Subsections \ref{T deltas subsection} and \ref{T derivatives subsection}).   We collect here some concepts that will be employed in the next sections. Our notation for manifolds is as in \cite{GPBook}, while we follow \cite{EKBook} for distribution theory and asymptotic expansions. In particular, we denote by $\langle f, \varphi \rangle$ the dual pairing between a distribution $f$ and a test function $\varphi$. 
 
We write $\mathrm{d} \sigma= \mathrm{d} \sigma_{\Sigma}$ for the volume density\footnote{As explained in  \cite[Section 10.3]{Folland84}, the measure $\mathrm{d}\sigma$ can always be constructed as a  volume \emph{density} for (not necessarily orientable) Riemannian manifolds. This is in contrast to the notion of a global volume form, which would require to impose orientability. } of $\Sigma$ (or any other submanifold depending on the context). 
 By $K\Subset \Omega$, we mean that $K$ is a compact subset of the interior of the set $\Omega$. It will be notationally convenient for us to employ two different symbols to distinguish, in some situation, distributional derivatives of functions from classical derivatives. We shall follow Farassat's convention from \cite{Fa96}. So, as usual, $\partial/\partial x_{i}$ stands for the classical partial derivative of functions, while $\bar{\partial}/\partial x_i$ stands for the distributional one, which is of course defined via the formula
 $$
 \Big\langle \frac{\bar{\partial} f}{\partial x_i}, \varphi \Big\rangle= -  \Big\langle f, \frac{\partial \varphi}{\partial x_i} \Big\rangle
, 
 $$
 where $\varphi$ is a test function.
 
 \subsection{The tubular neighborhood theorem}\label{subsection tubular neighborhood}
We shall make extensive use of the following result, known as the tubular neighborhood theorem. Given a positive continuous function $\epsilon:\Sigma\to(0,\infty)$, we shall call the open set
\[
\Sigma^{\epsilon}=\{x\in \mathbb{R}^{n}:\: \|x-\xi\| <\epsilon(\xi) \mbox{ for some }\xi\in\Sigma \}
\]
an $\epsilon$-neighborhood of $\Sigma$.

\begin{lemma}
\label{Tubular lemma} There is a continuous function
$\epsilon:\Sigma\to (0,\infty)$ such that every point $x\in \Sigma^{\epsilon}$ possesses a unique closest point\footnote{That is, a unique point $\pi(x)\in \Sigma$ such that $\operatorname*{dist}(x,\Sigma)=\|x-\pi(x)\|$.} $\pi(x)$ to the submanifold $\Sigma$ and the retraction $\pi: \Sigma^{\epsilon} \to \Sigma$ is a smooth submersion.
\end{lemma}

A proof of Lemma \ref{Tubular lemma} can be found e.g. in \cite[p.~69]{GPBook}. We recall that $\pi$ being a submersion means that $d\pi_{x}:\mathbb{R}^{n}\to T_{\pi(x)}(\Sigma)$ is surjective for all $x\in \Sigma^{\epsilon}$, where $T_{\pi(x)}(\Sigma)$ is the tangent space of $\Sigma$ at $\pi(x)$. We can actually give a more precise description of $d\pi_{\xi}$ for $\xi\in \Sigma$. We shall denote as $ N_{\xi}(\Sigma)$ the \emph{normal space} at $\xi\in\Sigma$, so that $\mathbb{R}^{n}= T_{\xi}(\Sigma)\oplus N_{\xi}(\Sigma)$. One may verify that $x-\pi(x)\in N_{x}(\Sigma)$ for each $x\in \Sigma^{\epsilon}$.

\begin{lemma} 
\label{Tlemmaprojection}  We have $N_{\pi(x)}(\Sigma) \subseteq \operatorname*{ker} [(d\pi_{x})^{\top}]$ for all $x\in\Sigma^{\epsilon}$. Furthermore, $d\pi_\xi$ is the orthogonal projection onto $T_{\xi}(\Sigma)$ for each $\xi\in\Sigma$. 

\end{lemma} 
\begin{proof}
Let $v\in N_{\pi(x)}(\Sigma)$ and let $w\in\mathbb{R}^{n}$ be arbitrary. Since $d\pi_{x}(w)\in T_{\pi(x)}(\Sigma)$, we obtain $d\pi_{x}^{\top}(v) \cdot w= v \cdot d\pi_{x}(w)=0$, and we deduce that $d\pi_{x}^{\top}(v)=0$. For the second claim, 
 observe that if $v$ belongs to a bounded subset of $N_{\xi}(\Sigma)$ and $t$ is small enough, then $\pi(\xi+tv)$ is a constant function of $t$, so that
\begin{equation*}
d\pi_\xi(v)= \lim_{t\to 0} \frac{\pi(\xi+tv)-\pi(\xi)}{t}=0.
\end{equation*}
Therefore, $N_{\xi}(\Sigma) \subseteq \operatorname*{ker} d\pi_{\xi}.$ On the other hand, since $\pi$ is the identity on $\Sigma$,  the linear map $(d\pi_\xi)_{|T_{\xi}(\Sigma)}$ is the identity as well, whence the result follows.
\end{proof}
 
 The relevance of $\Sigma^{\epsilon}$ for us is that it is diffeomorphic to an open neighborhood of $\Sigma$ inside its normal bundle (see \cite{GPBook}).

\subsection{$\delta$-derivatives} Let $f$ be a smooth function only defined on $\Sigma$. An expression such as $\partial f/\partial x_i$ makes  no sense in general. Nevertheless, sometimes it  is convenient, and even necessary, to work with substitutes of the partial derivatives for $f$ with respect to the variables of the surrounding space $\mathbb{R}^{n}$. 

We shall consider here a slight extension of the so-called $\delta$-derivatives, usually defined for hypersurfaces  \cite{EK85,EKBook}.  Let $F$ be a smooth extension of $f$ to an open neighborhood of $\Sigma$ in $\mathbb{R}^{n}$. We define the $\delta$-derivative of $f$ with respect to the variable  $x_i$ at the point $\xi\in\Sigma$ as the $i$-th coordinate of the projection of the gradient $\left.\nabla F\right|_ {\xi}$ onto the tangent space $T_{\xi}(\Sigma)$. If $\{\mathbf{n}_{1}, \mathbf{n}_{2}, \dots, \mathbf{n}_{d}\}$ is a local smooth frame on an open subset $W\subset \Sigma$ that is an orthonormal basis of $N_{\xi}(\Sigma)$ at each $\xi\in W$, we have
\[
\frac{\delta f}{\delta x_i}= \left.\frac{\partial F}{\partial x_i}\right|_{W} - \sum_{k=1}^{d} n_{k,i} \frac{d F}{d \mathbf{n}_k} 
\]
where ${d F}/{d \mathbf{n}_k}= \mathbf{n}_k \cdot \nabla F\left.\right|_{\Sigma}$ is the normal derivative in the direction of $\mathbf{n}_k$ 
and $n_{k,i}$ is the $i$-th coordinate of $\mathbf{n}_k$. 
One might verify that $\delta f/\delta x_j$ is independent of the choice of the extension $F$.

A particularly useful extension of $f$  is provided by the function $f\circ \pi$ defined on the tubular neighborhood $\Sigma^{\epsilon}$. If $v \in N_{\xi} (\Sigma)$, we have
$$
\left.\nabla (f\circ \pi)\right|_{\xi} \cdot v= (d f_{\xi} \circ d\pi_{\xi})(v)=0,
$$
because of Lemma \ref{Tlemmaprojection}. This means that $\left.\nabla (f\circ \pi)\right|_{\xi}\in T_{\xi}(\Sigma)$ at each $\xi\in\Sigma$, which yields the neat expression
\begin{equation}
\label{Teq2.1} \frac{\delta f}{\delta x_i}= \left.\frac{\partial( f\circ \pi)}{\partial x_i}\right|_{\Sigma}.
\end{equation}

\subsection{Strong asymptotic expansions}
We shall need the concept of \emph{strong} asymptotic expansions. It is well known that asymptotic expansions cannot be differentiated in general, but if they may, we say that they hold strongly. More precisely, if $b$ and $c_j$ are smooth functions, we say that $b$  admits the strong asymptotic expansion
$$
b(x)\sim \sum_{j} c_{j}(x)
$$
if for each multi-index $\alpha$
$$
\frac{\partial^{\alpha}b}{\partial x^{\alpha}}\sim \sum_{j} \frac{\partial^{\alpha}c_{j}}{\partial x^{\alpha}}\: .
$$
\section{The space of test functions}\label{space test functions}
We now introduce our new spaces of test functions. We shall fix 
an $\epsilon$-neighborhood  $\Sigma^{\epsilon}$ such that $\pi: \Sigma^{\epsilon}\to \Sigma$ is a smooth retractive submersion as stated in Lemma \ref{Tubular lemma}.  Our construction makes use of the following lemma. 

\begin{lemma}
\label{lemma1} There is a family $\{U_{\nu}\}_{\nu\in \mathbb{N}}$ of bounded open subsets of $\Sigma^{\epsilon}$ such that $\{\Sigma\cap U_{\nu}\}_{\nu\in \mathbb{N}}$ is an open covering of $\Sigma$ and  each $U_{\nu}$ has the following properties: 
\begin{itemize}
\item [(i)] $U_{\nu}\cap \Sigma \neq \emptyset$;
\item [(ii)] there are smooth vector fields $\mathbf{n}^{\nu}_{j}:U_{\nu}\cap \Sigma\to \mathbb{S}^{n-1}$, $j=1,\dots, d,$ such that the frame $\{\mathbf{n}^{\nu}_{1}(\xi), \mathbf{n}_{2}^{\nu}(\xi), \dots, \mathbf{n}^{\nu}_{d}(\xi)\}$ forms an orthonormal basis of the normal space $N_{\xi}(\Sigma)$ at each $\xi\in \Sigma$;
\item [(iii)] there is $\rho_\nu>0$ such that the map $(\xi, \omega, \rho)\to \xi+  \rho \sum_{j=1}^{d}\omega_j \mathbf{n}_{j}^{\nu}(\xi)$, where $\omega=(\omega_1,\dots, \omega_{d})$, is a diffeomorphism from $(\Sigma\cap U_{\nu})\times \mathbb{S}^{d-1} \times (0,\rho_{\nu})$ onto $U_{\nu}\setminus{\Sigma}$.
\end{itemize}

\end{lemma}
\begin{proof}
The proof is standard, but we include it for the sake of completeness. Using local coordinates, one finds an open covering $\{W_{\nu}\}_{\nu\in\mathbb{N}}$ of $\Sigma$ with each $\emptyset\neq W_{\nu}\subset\Sigma$ being bounded and corresponding vector fields $\mathbf{n}^{\nu}_{j}:W_{\nu}\to \mathbb{S}^{n-1}$, $j=1,\dots, d,$ such that $\{\mathbf{n}^{\nu}_{1}(\xi), \mathbf{n}_{2}^{\nu}(\xi), \dots, \mathbf{n}^{\nu}_{d}(\xi)\}$ is an orthonormal basis of $N_{\xi}(\Sigma)$ at each $\xi\in \Sigma$. We pick $0<\rho_{\nu}< \inf_{\xi\in W_{\nu}} \epsilon(\xi)$ and set $U_{\nu}=\{x\in \Sigma_{\epsilon}: \: \pi(x)\in W_{\nu} \mbox{ and } \|x-\pi(x)\|<\rho_{\nu} \}$, so that (i) and (ii) trivially hold. If we choose $\rho_{\nu}$ sufficiently small, we can guarantee that the map $\Phi_{\nu}(\xi,\omega,\rho)=  \xi+  \rho \sum_{j=1}^{d}\omega_j \mathbf{n}_{j}^{\nu}(\xi)$ is a smooth bijection $\Phi_{\nu}:W_{\nu}\times \mathbb{S}^{d-1} \times (0,\rho_{\nu})\to U_{\nu}\setminus \Sigma$ .  Its inverse is the smooth function
\begin{equation}
\label{eqTest1}
\Phi^{-1}_{\nu}(x)=(\pi(x), \omega_{x}, \|x-\pi(x)\|),
\end{equation}
where the  $k$-th coordinate of the unit vector $\omega_{x}\in\mathbb{S}^{d-1}$ is
given by
\begin{equation}
\label{teq2}
\omega_{x,k}=\frac{\mathbf{n}_{k}^{\nu}(\pi(x)) \cdot (x-\pi(x))}{\|x-\pi(x)\|}.
\end{equation}

\end{proof}

\begin{figure}[ptbh]
\centerline{\includegraphics[scale=0.8]
{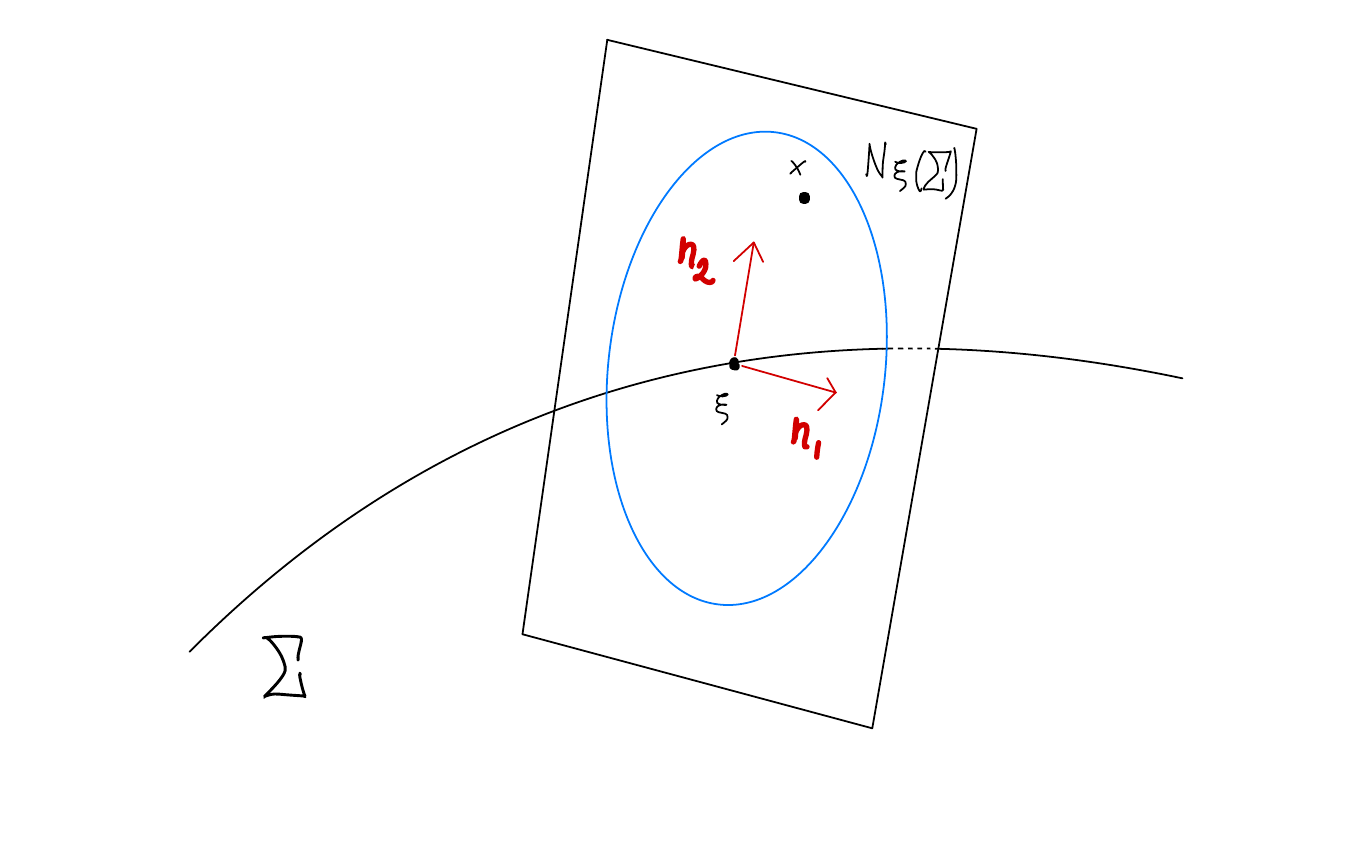}}
\caption{ Local tube coordinates $x=\xi+ \rho \omega_1 \mathbf{n}_1 +\rho \omega_2 \mathbf{n}_2$ for a curve $\Sigma$ in $\mathbb{R}^3$}
\label{fig}\end{figure}

From now on, we shall also fix a family  $\{U_{\nu}\}_{\nu\in \mathbb{N}}$ having the properties stated in Lemma \ref{lemma1}.
To ease our writing, we denote as $(\xi_{x},\omega_{x}, \rho_{x})$ the image of $x\in U_{\nu}$ under the inverse of the diffeomorphism from Lemma \ref{lemma1}(iii). Naturally, $(\xi_{x},\omega_{x}, \rho_{x})$ has the explicit form \eqref{eqTest1}; see Figure \ref{fig} for a graphic representation of the local tube coordinates.

\begin{definition}\label{def3.2}
Let $\Omega$ be an open subset of $\mathbb{R}^{n}$ such that $\Omega\cap \Sigma\neq 0$. 
\begin{itemize}
\item [(i)]
We define $\mathcal{D}_{\ast, \Sigma}(\Omega)$ as the function space consisting of all $\phi\in C^{\infty}(\Omega\setminus{\Sigma})$ such that on each $\Omega\cap U_{\nu}\setminus\Sigma\neq \emptyset$ the function $\phi$ admits (uniform) \emph{strong} asymptotic expansion
\begin{equation}
\label{teq1}
\phi(x)\sim \sum_{j=m}^{\infty} a_{j}(\xi_{x}, \omega_{x}) \: \rho_{x}^{j} \qquad \mbox{as }\rho_{x}\to0,
\end{equation}
for some $m\in\mathbb{Z}$ and some smooth functions $a_{j}=a_{j}^{\phi,\nu} \in C^{ \infty}((\Sigma \cap U_{\nu})\times \mathbb{S}^{d-1})$.
\item [(ii)]  We write $\mathcal{D}^{[m;K]}_{\ast, \Sigma}(\Omega)$  for the subspace of $\mathcal{D}_{\ast, \Sigma}(\Omega)$ consisting of functions $\phi$ that vanish outside a given compact set $K\Subset \Omega$ and have strong asymptotic expansion \eqref{teq1} starting at a given $m\in\mathbb{Z}$.
\item [(iii)] $\mathcal{E}_{\ast, \Sigma}(\Omega)$ stands for the space of all $\psi\in C^{\infty}(\Omega\setminus{\Sigma})$ such that $\varphi \cdot \psi \in \mathcal{D}_{\ast, \Sigma}(\Omega)$ for every $\varphi \in \mathcal{D}(\Omega)$.
\end{itemize}
\end{definition}

\smallskip

The spaces $\mathcal{D}_{\ast, \Sigma}(\Omega)$ and $\mathcal{E}_{\ast, \Sigma}(\Omega)$ are closed under differentiation:

\begin{proposition}
\label{tp1} For each $\alpha\in\mathbb{N}^{n}$, we have $(\partial/\partial x)^{\alpha}: \mathcal{D}^{[m;K]}_{\ast, \Sigma}(\Omega)\to \mathcal{D}^{[m-|\alpha|;K]}_{\ast, \Sigma}(\Omega) $.
\end{proposition}
\begin{proof}
Since our assumption is that \eqref{teq1} can be differentiated termwise, it suffices to study terms of the form 
\begin{align*}
\frac{\partial}{\partial x_{i} } (\rho_{x}^{j} a(\xi_{x},\omega_{x}))=j\rho_{x}^{j-1} a(\xi_{x},\omega_{x}) \frac{\partial \rho_{x}}{\partial x_{i} }+ \rho_{x}^{j}\frac{\partial}{\partial x_{i} } (a(\xi_{x},\omega_{x})), \qquad i=1, \dots, n,
\end{align*}
where $a\in C^{\infty}((\Sigma\cap U_{\nu}) \times \mathbb{S}^{d-1} )$. 

Let us first show that $\nabla \rho_{x} $, and consequently each $\partial \rho_{x}/\partial x_{i}$, only depends on $(\xi_{x},\omega_{x})$ and not on $\rho_{x}$. We verify this with the aid of Lemma \ref{Tlemmaprojection} and the chain rule. In fact, using
$x-\pi(x) \in N_{\pi(x)}(\Sigma)$,
		
			\begin{align*}
				(\nabla \rho |_x)^{\top}&=\left[\nabla (\|\cdot\|)|_{x-\pi(x)} \cdot d(x-\pi(x))_{x}\right]^{\top}\\
				&=\frac{x-\pi(x)}{\|x-\pi(x)\|}- d\pi_x^{\top}\left(\frac{x-\pi(x)}{\|x-\pi(x)\|}\right)\\
				&=\frac{x-\pi(x)}{\|x-\pi(x)\|}=\sum_{k=1}^{d}\omega_{x,k} \mathbf{n}_k^{\nu}(\xi_{x}),
			\end{align*}
where $\omega_{x,k}$ is given by \eqref{teq2}.

Next, we smoothly extend $a$ as $A(\xi,\omega)=a(\pi(\xi), \omega/\|\omega\|)$ for each $(\xi,\omega)\in U_{\nu} \times (\mathbb{R}^{d}\setminus\{0\}).$ We set $\pi(x)=\xi_{x}=(\xi_{1}(x), \dots, \xi_{n}(x)) $. We also write $\omega_{x}=(\omega_{x,1},\dots, \omega_{x,n})$. Hence (see \eqref{Teq2.1}), 
\begin{align*}
\frac{\partial}{\partial x_{i} } (a(\xi_{x},\omega_{x}))&=\sum_{l=1}^{n} \left.\frac{\partial A}{\partial \xi_{l} }\right|_{(\xi_{x},\omega_x)}\left.\frac{ \partial \xi_{l}}{\partial x_i}\right|_{ x} + \sum_{k=1}^{d}\left.\frac{\partial A}{\partial \omega_{k} }\right|_{(\xi_{x},\omega_{x})}\frac{ \partial}{\partial x_i}\left(\frac{\mathbf{n}_{k}^{\nu}(\xi_x) \cdot (x-\xi_{x})}{\rho_{x}}\right)\\
&
= \sum_{l=1}^{n}\left. \frac{\delta a}{\delta \xi_{l} }\right|_{(\xi_{x},\omega_x)}\left.\frac{\partial \xi_{l}}{\partial x_i}\right|_{ x }
+\sum_{k=1}^{d} \frac{1}{\rho_x} \left. \frac{\delta a}{\delta \omega_{k} }\right|_{(\xi_{x},\omega_{x})}\left(- {\omega_{x,k}}\left.\frac{\partial \rho_x}{\partial x_i}\right|_{(\xi_{x},\omega_{x})} \right.
\\
&
\qquad  \qquad
\left.+\mathbf{n}_{k}^{\nu}(\xi_x) \cdot\frac{ \partial}{\partial x_i} (x-\xi_{x})+\left.\frac{\delta \mathbf{n}^{\nu}_{k}}{\delta x_{i}}\right|_{\xi_x} \cdot  (x-\xi_{x})
\right).
\end{align*}
Let $\mathbf{n}^{\nu}_{k}=(n^{\nu}_{k,1}, \dots, n^{\nu}_{k,n})$. The term $\mathbf{n}_{k}^{\nu}(\xi_x) \cdot(\partial/\partial x_i) (x-\xi_{x})$ can be simplified. It is the $i$-th component of the vector
$$
(d(x-\xi(x))_x)^{\top} (\mathbf{n}^{\nu}_{k}) = \mathbf{n}^{\nu}_{k} - d\pi_x^{\top}(\mathbf{n}^{\nu}_{k})= \mathbf{n}^{\nu}_{k},
$$
so that $\mathbf{n}_{k}^{\nu}(\xi_x) \cdot(\partial/\partial x_i) (x-\xi_{x})= n^{\nu}_{k,i}(\xi_x)$. Each function $\partial \xi_{l}/\partial x_{i}$ is smooth on $\Sigma^{\epsilon}$, and we may then use Lemma \ref{lemma1}(iii) and Taylor's theorem to expand it as (see the proof of Theorem \ref{tth1} below)
\begin{equation}
\label{eqdef b coefficient}
\left.\frac{ \partial \xi_{l}}{\partial x_i}\right|_{ x}\sim \left.\frac{ \partial \xi_{l}}{\partial x_i}\right|_{ \xi_{x}} + \sum_{q=1}^{\infty} \rho_{x}^{q} \: b_{l,i,q}(\xi_{x},\omega_{x}) \qquad  \mbox{as }\rho_{x}\to0,
\end{equation}
where each $b_{l,i,q}\in C^{\infty}((\Sigma\cap U_{\nu}) \times \mathbb{S}^{d-1} ).$ Noticing that
\[
\sum_{l=1}^{n}\left. \frac{\delta a}{\delta \xi_{l} }\right|_{(\xi_{x},\omega_x)}\left.\frac{\partial \xi_{l}}{\partial x_i}\right|_{ \xi_{x} }= \left. \frac{\delta a}{\delta \xi_{i} }\right|_{(\xi_{x},\omega_x)}
\]
because $(\delta a/\delta\xi_1, \dots, \delta a/\delta\xi_n)$ is a tangent vector and $d\pi_{\xi_x}$ is the orthogonal projection onto the tangent space, 
we finally obtain
\begin{align*}
\frac{\partial}{\partial x_{i} } (a(\xi_{x},\omega_{x}))&\sim  \frac{1}{\rho_x} \sum_{k=1}^{d} \left.  \frac{\delta a}{\delta \omega_{k} }\right|_{(\xi_{x},\omega_{x})}\left(n_{k,i}^{\nu}(\xi_x)- {\omega_{x,k}}\left.\frac{\partial \rho_x}{\partial x_i}\right|_{(\xi_{x},\omega_{x})}\right)
\\
&
\qquad  \qquad
+ \left. \frac{\delta a}{\delta \xi_{i} }\right|_{(\xi_{x},\omega_x)}
+\sum_{k=1}^{d}\sum_{h=1}^{d}\omega_{x,h} \left.  \frac{\delta a}{\delta \omega_{k} }\right|_{(\xi_{x},\omega_{x})}\mathbf{n}^{\nu}_{h}(\xi_x)\cdot \left.\frac{\delta \mathbf{n}^{\nu}_{k}}{\delta x_{i}}\right|_{\xi_x}
\\
&
\qquad  \qquad
+ \sum_{q=1}^{\infty} \rho_{x}^{q} \sum_{l=1}^{n}b_{l,i,q}(\xi_{x},\omega_{x})\left. \frac{\delta a}{\delta \xi_{l} }\right|_{(\xi_{x},\omega_x)}  \qquad  \mbox{as }\rho_{x}\to0.
\end{align*}
The result follows from this asymptotic expansion formula.
\end{proof}
For future reference, let us collect in a corollary the actual formula that we have just established in the proof of Proposition \ref{tp1}.
\begin{corollary}
\label{tcfpartialdefivatives}
If $a\in C^{\infty}((\Sigma\cap U_{\nu}) \times \mathbb{S}^{d-1} )$ and $\mathbf{n}^{\nu}_{k}=(n^{\nu}_{k,1}, \dots, n^{\nu}_{k,n})$, then 
\begin{align*}
\frac{\partial}{\partial x_{i} } (\rho_{x}^{j} a(\xi_{x},\omega_{x}))&\sim \rho_{x}^{j-1}\left( j a(\xi_{x},\omega_{x})\sum_{l=1}^{d}\omega_{x,l} n^{\nu}_{l,i}
+\sum_{k=1}^{d}  \frac{\delta a}{\delta \omega_{k} }\left(n_{k,i}^{\nu}- {\omega_{x,k}}\sum_{h=1}^{d}\omega_{x,h} n^{\nu}_{h,i}\right)\right)
\\
&
\qquad \qquad 
+ \rho_{x}^{j}\left( \frac{\delta a}{\delta \xi_{i} }
+\sum_{k=1}^{d}\sum_{h=1}^{d}\omega_{x,h}  \frac{\delta a}{\delta \omega_{k} }\mathbf{n}^{\nu}_{h}\cdot \frac{\delta \mathbf{n}^{\nu}_{k}}{\delta x_{i}} \right)
\\
&
  \qquad   \qquad 
 + \sum_{q=j+1}^{\infty} \rho_{x}^{q} \sum_{l=1}^{n}b_{l,i,q-j}(\xi_{x},\omega_{x}) \frac{\delta a}{\delta \xi_{l} } \qquad  \mbox{as } \rho_{x}\to0,
\end{align*}
where the coefficient functions $b_{l,i,q-j} \in C^{\infty}((\Sigma\cap U_{\nu}) \times \mathbb{S}^{d-1} )$ are given by \eqref{eqdef b coefficient}.

\end{corollary}
If $\phi\in \mathcal{D}_{\ast, \Sigma}(\Omega)$ has the strong asymptotic expansion \eqref{teq1} on $\Omega\cap U_{\nu}\setminus\Sigma$, we denote by $a_{j,\alpha}=a_{j,\alpha}^{\phi,\nu}\in C^{\infty}(\Omega\cap U_{\nu}\setminus \Sigma)$ the $j$-th coefficient of the asymptotic expansion of $\partial^{\alpha} \phi /\partial x^{\alpha}$, so that
$$
\frac{\partial^{\alpha}\phi}{\partial x^{\alpha}}\sim \sum_{j=m-|\alpha|}^{\infty} a_{j,\alpha}(\xi_{x}, \omega_{x}) \: \rho_{x}^{j} \qquad \mbox{as }\rho_{x}\to0.
$$

We now topologize $\mathcal{D}_{\ast, \Sigma}(\Omega)$. We set $U=\bigcup_{\nu\in\mathbb{N}} U_{\nu}$. We first give to each $\mathcal{D}^{[m;K]}_{\ast, \Sigma}$ the structure of a Fr\'{e}chet space via the family of seminorms 
\[
\|\phi\|_{q,k,\nu,K,m} = \sup_{\underset {|\alpha|\leq k}{x\in K \cap U_{\nu}} } \rho_{x}^{-q}\Big|(\partial/\partial x)^{\alpha}\phi (x) -\sum_{j=m-|\alpha|}^{q-1} a_{j,\alpha}(\xi_{x}, \omega_{x}) \: \rho_{x}^{j} \Big|+ \sup_{\underset {|\alpha|\leq k}{x\in K \setminus U}} |(\partial/\partial x)^{\alpha} \phi (x) |,
\]
where $q>m$ and $k,\nu\in\mathbb{N}$. The locally convex space topology of $\mathcal{D}_{\ast, \Sigma}(\Omega)$ is then 
that of a strict inductive limit of Fr\'{e}chet spaces (LF-space) \cite{TBook} as the inductive union
$$
\mathcal{D}_{\ast, \Sigma}(\Omega)=\bigcup_{K\Subset \Omega,m\in\mathbb{Z}}\mathcal{D}^{[m;K]}_{\ast, \Sigma}=\varinjlim_{K\Subset \Omega,m\in\mathbb{Z}} \mathcal{D}^{[m;K]}_{\ast, \Sigma}.
$$
The following theorem collects some of the main properties of  $\mathcal{D}_{\ast, \Sigma}(\Omega)$.
\begin{theorem}
\label{tth1} $\mathcal{D}_{\ast, \Sigma}(\Omega)$ is a Montel space. Each partial derivative 
$$\frac{\partial^{\alpha}}{\partial x^{\alpha}}:\mathcal{D}_{\ast, \Sigma}(\Omega)\to \mathcal{D}_{\ast, \Sigma}(\Omega)
$$
is a continuous operator.  Moreover, the Schwartz test function space $\mathcal{D}(\Omega)$ is a closed subspace of $\mathcal{D}_{\ast, \Sigma}(\Omega)$. 
\end{theorem}
\begin{proof}
Since every strict LF-space is regular \cite{TBook}, it suffices to show that each $\mathcal{D}^{[m;K]}_{\ast, \Sigma}$ is an FS-space (Fr\'{e}chet-Schwartz space). The latter might directly be deduced from the Arzel\`{a}-Ascoli theorem via a classical argument  \cite{TBook}; we therefore leave the details to the reader. The continuity of the partial differential operators follows from Proposition \ref{tp1} and the definition of the seminorms $\| \ \: \|_{q,k,\nu,K,m}$. For the last claim, we localize. Similarly to Lemma \ref{lemma1}, we define the diffeomorphism $\Psi_{\nu}: (\Sigma\cap U_{\nu})\times B_{\rho_\nu}^{d}\to U_{\nu} $ given by $\Psi_{\nu}(\xi,y)=\xi + \sum_{k=1}^{d} y_k \mathbf{n}_{k}^{\nu}(\xi)$, where $B_{\rho_\nu}^{d}=\{y\in\mathbb{R}^{d}: \: \|y\|<\rho_\nu \}$. If $\varphi \in\mathcal{D}(\Omega)$, then $\varphi\circ \Psi_{\nu}$ is smooth on $(\Sigma\cap U_{\nu})\times B_{\rho_\nu}^{d}$, and, by Taylor's theorem applied to $\varphi\circ \Psi_{\nu}$,
\begin{equation}
\label{Teq3.3}
\varphi(x)\sim \sum_{j=0}^{\infty}  \rho_{x}^{j} \: \sum_{|\alpha|=j}\left.\frac{\partial^{\alpha}(\varphi\circ \Psi_{\nu})}{\partial y^{\alpha}} \right|_{(\xi_{x},0)} \frac{\omega_{x}^{\alpha}    
}{\alpha!}
\qquad  \mbox{as } \rho_{x}\to 0.
\end{equation}
Conversely, any element $\varphi \in \mathcal{D}_{\ast, \Sigma}(\Omega)$ that has strong asymptotic expansion
\begin{equation}
\label{teqExpD}
\varphi(x)\sim \sum_{j=0}^{\infty}  \rho_{x}^{j}\sum_{|\alpha|=j} c_{\alpha}(\xi_{x}) \:\omega_{x}^{\alpha}  \qquad  \mbox{as } \rho_{x}\to 0
\end{equation}
on each $U_{\nu}\setminus\Sigma$
for some smooth functions $c_{\beta}=c_{\beta,\nu}\in C^{\infty}(U_{\nu}\cap\Sigma)$
must admit a smooth extension to $\Omega$ and therefore belongs to $\mathcal{D}(\Omega)$. These observations imply that $\mathcal{D}(\Omega)\subset \mathcal{D}_{\ast, \Sigma}(\Omega)$  
 is a closed subspace (and that its usual canonical locally convex structure coincides with the relative topology inhered from $\mathcal{D}_{\ast, \Sigma}(\Omega)$).
 \end{proof}
 
We conclude this section discussing some properties of the space $\mathcal{E}_{\ast, \Sigma}(\Omega)$. We start by pointing out that $\mathcal{E}_{\ast, \Sigma}(\Omega)$ coincides with the space of multipliers of $\mathcal{D}_{\ast, \Sigma}(\Omega)$ (also known as its Moyal algebra). In fact, the reader can readily verify that given $\psi\in\mathcal{E}_{\ast, \Sigma}(\Omega)$, the multiplication  operators 
\begin{equation}
\label{teqmultip}M_{\psi}:\mathcal{D}_{\ast, \Sigma}(\Omega)\to\mathcal{D}_{\ast, \Sigma}(\Omega), \qquad M_{\psi}(\phi)= \psi \cdot \phi,
\end{equation}
are automatically continuous. 

 This space can also be provided with a natural topological vector space structure, but its locally convex topology is more complex than that of $\mathcal{D}_{\ast, \Sigma}(\Omega)$. Let $K\Subset \Omega$ be regular (i.e., $\operatorname*{int}K\neq \emptyset$) and let $\mathcal{E}^{[m;K]}_{\ast, \Sigma}$ be the Fr\'{e}chet space consisting of all those $\phi\in C^{\infty}(K\setminus \Sigma)$ such that their restrictions to each $U_{\nu}\cap K \setminus \Sigma$ have strong asymptotic expansions as in \eqref{teq1}. Its canonical Fr\'{e}chet space structure is generated by the family of seminorms $\{\|\ \: \|_{q,k,\nu,K,m}\}$. One might then check $\mathcal{E}_{\ast, \Sigma}(\Omega)=\bigcap_{K\Subset\Omega
} \bigcup_{m\in \mathbb{Z}}\mathcal{E}^{[m;K]}_{\ast, \Sigma}$, which yields the definition of a natural locally convex topology on it:
\[
\mathcal{E}_{\ast, \Sigma}(\Omega)=\varprojlim_{K\Subset\Omega
} \varinjlim_{m\in \mathbb{Z}}\mathcal{E}^{[ m;K]}_{\ast, \Sigma}\: .
\]

 \section{Distributions with a thick submanifold}\label{thick distributions section}
 We are ready to define our new distribution space having $\Sigma$ as a thick submanifold.
 
 \begin{definition}  \label{definition 2}  Let $\Omega\subset \mathbb{R}^{n}$ be an open subset such that $\Omega\cap\Sigma\neq \emptyset$. The space of distributions on $\Omega$ with thick submanifold $\Sigma$ is the dual space of $\mathcal{D}_{\ast,\Sigma}(\Omega)$ provided with the strong dual topology. We denote it by $\mathcal{D}'_{\ast,\Sigma}(\Omega)$ and call its elements \emph{thick distributions along} $\Sigma$ (or simply $\Sigma$-thick distributions).

 \end{definition}

 In view of Theorem \ref{tth1} (and \cite[Proposition 24.25]{MVBook}), the thick distribution space $\mathcal{D}'_{\ast,\Sigma}(\Omega)$ is Montel. We denote as $\Pi$ the transpose of the inclusion map  $\iota: \mathcal{D}(\Omega)\to \mathcal{D}_{\ast,\Sigma}(\Omega)$. It is a well-behaved projection.
 
 \begin{proposition} The map $\Pi: \mathcal{D}_{\ast,\Sigma}'(\Omega)\to \mathcal{D}'(\Omega)
 $  is an open continuous surjection.
 \end{proposition}
\begin{proof} The continuity follows by definition. Theorem \ref{tth1} and the Hahn-Banach theorem ensure that $\Pi$ is surjective. Since $\mathcal{D}'(\Omega)$ is ultrabornological \cite[Theorem 3.2]{D2004} and $\mathcal{D}'_{\ast,\Sigma}(\Omega)$ is webbed \cite[Theorem 14.6.5]{NBBook}, De Wilde's open mapping theorem \cite[p. 481]{NBBook} applies and therefore $\Pi$ is open.
\end{proof}

Theorem \ref{tth1} also allows us to talk about partial derivative operators on $\mathcal{D}_{\ast,\Sigma}'(\Omega)$, which we shall denote as $\partial_{\ast}^{\alpha}/\partial x^{\alpha}$. In fact, as in the case of the classical distributional derivatives, we define the latter as the transpose of $(-1)^{|\alpha|}\partial^{\alpha}/\partial x^{\alpha}$, that is, given $f\in\mathcal{D}_{\ast,\Sigma}'(\Omega)$, we have that the action of $(\partial_{\ast}^{\alpha}/\partial x^{\alpha})f\in\mathcal{D}_{\ast,\Sigma}'(\Omega)$ at $\phi \in\mathcal{D}_{\ast,\Sigma}(\Omega)$ is given by
\[
\Big\langle \frac{\partial_{\ast}^{\alpha}f}{\partial x^{\alpha}},\phi \Big\rangle= (-1)^{|\alpha|} \Big\langle f, \frac{\partial^{\alpha}\phi}{\partial x^{\alpha}}
\Big\rangle.
\]

Given $\psi\in \mathcal{E}_{\ast,\Sigma}(\Omega)$, we can extend the definition of multiplication by $\psi$ to $\mathcal{D}'_{\ast,\Sigma}(\Omega)$ as the transpose of \eqref{teqmultip}, that is, 
\[  \langle \psi f, \phi \rangle= \langle f, \psi \phi \rangle,
\]
for each test function $\phi\in \mathcal{D}_{\ast,\Sigma}(\Omega)$.

The next facts directly follow from the definitions, but we state them as a proposition in order to highlight their importance for the theory of thick distributions.

\begin{proposition}
\label{td proposition differentiation}
For each $f\in\mathcal{D}_{\ast,\Sigma}'(\Omega)$ and $\alpha\in\mathbb{N}^{n}$, we have
\[
\Pi\left(\frac{\partial_{\ast}^{\alpha}f}{\partial x^{\alpha}}\right)= \frac{\bar{\partial}^{\alpha}\Pi(f)}{\partial x^{\alpha}}.
\]
Furthermore, if $\psi \in \mathcal{E}_{\ast,\Sigma}(\Omega)$, then 
\[ \frac{\partial^{\alpha}_{\ast}(\psi f)}{\partial x^{\alpha}}= \sum_{\beta\leq \alpha} \binom{\alpha}{\beta}  \frac{\partial^{\beta}\psi }{\partial x^{\beta}}\ \frac{\partial^{\alpha-\beta}_{\ast}f}{\partial x^{\alpha-\beta}}\: .
\]
\end{proposition} 

We can also perform certain changes of variables. Indeed, let $\Psi: \tilde{\Omega}\to  \Omega$ be a diffeomorphism that maps $\tilde{\Omega}\cap\tilde{\Sigma}$  onto $\Omega\cap\Sigma$, where $\tilde{\Sigma}$ is another $d$-codimensional closed submanifold of $\mathbb{R}^{n}$. Naturally, we always have $d\Psi_{\xi}(T_{\xi} (\tilde{\Sigma}))= T_{\Psi(\xi)}(\Sigma)$. If we additionally require
\begin{equation}
\label{teqdiffcv}
d\Psi_{y}(N_{y}( \tilde{\Sigma}))= N_{\Psi(y)}(\Sigma),  \qquad y\in\tilde{\Omega}\cap \tilde{\Sigma},
\end{equation}
for the normal spaces, then the pullback map $\Psi^{\ast}: \mathcal{D}_{\ast,\Sigma}(\Omega)\to \mathcal{D}_{\ast,\tilde{\Sigma}}(\tilde{\Omega})$ is an isomorphism of locally convex spaces, where as usual $\Psi^{\ast}\phi =\phi\circ \Psi$. We can also pullback $\Sigma$-thick distributions along the diffeomorphism, so that $\Psi^{\ast}: \mathcal{D}'_{\ast,\Sigma}(\Omega)\to \mathcal{D}'_{\ast,\tilde{\Sigma}}(\tilde{\Omega})$, where we now define $\Psi^{\ast}f$ in the canonical way, namely,
\[\langle \Psi^{\ast} f, \varphi \rangle= \langle f(\Psi(y)), \varphi(y) \rangle= \Big\langle f(x), \frac{\varphi(\Psi^{-1}(x))}{|\det (d\Psi_{\Psi^{-1}(x)})|} \Big\rangle= \langle f, (\Psi^{-1})^{\ast}\left( \varphi \cdot |\det (d\Psi)|^{-1}\right) \rangle.
\]

We end this section with two important remarks.

\begin{remark}\label{trmSheaf} It is useful to extend the definition of $\mathcal{D}'_{\ast,\Sigma}$ to \emph{any} open subset $\Omega\subseteq \mathbb{R}^{n}$ as follows:
\begin{equation}
\label{teqsheaf}
\mathcal{D}'_{\ast,\Sigma}(\Omega)= 
\begin{cases}
\mathcal{D}'_{\ast,\Sigma}(\Omega) & \mbox{if }\Omega\cap\Sigma \neq \emptyset, \\
\mathcal{D}'(\Omega) & \mbox{if }\Omega\cap\Sigma= \emptyset.
\end{cases}
\end{equation}
Under this convention $\mathcal{D}'_{\ast,\Sigma}: \Omega \to \mathcal{D}'_{\ast,\Sigma}(\Omega)$ becomes a fine sheaf of locally convex spaces on $\mathbb{R}^{n}$, as follows from the existence of usual smooth partitions of unity. In particular, one may always speak about $\operatorname*{supp}f$, the support of a $\Sigma$-thick distribution $f\in\mathcal{D}'_{\ast,\Sigma}(\Omega)$. The projection operator $\Pi$ is local, namely, $\operatorname*{supp}\Pi(f)\subseteq\operatorname*{supp}f$. Actually, \eqref{teqsheaf} implies that $(\operatorname*{supp}f) \setminus (\operatorname*{supp}\Pi(f)) \subseteq \Sigma$.

Notice that the space of compact sections of $\mathcal{D}'_{\ast,\Sigma}(\Omega)$ coincides with $\mathcal{E}'_{\ast,\Sigma}(\Omega)$, exactly in the same way as $\mathcal{E}'(\Omega)$ is identified with the space of compactly supported distributions on $\Omega$.
\end{remark}
\begin{remark}\label{trmproj} If $f\in \operatorname*{ker}\Pi$, we must necessarily have that $\operatorname*{supp}f\subseteq \Sigma$. In order to determine which $\Sigma$-thick distributions supported on $\Sigma$ belong to $\operatorname*{ker}\Pi$, let us fix a smooth partition of unity $\{\chi_{\nu}\}_{\nu\in\mathbb{N}}$ on $\Sigma$ that is subordinate to $\{\Omega\cap \Sigma \cap U_{\nu}\}_{\nu\in\mathbb{N}}$. Localizing and changing variables, we first notice that $\operatorname*{supp}f\subseteq\Sigma$ if and only if, for each $\nu$, there is a sequence of distributions  
$\{f_{\nu, j}\}_{  j\in \mathbb{Z}}$ of $\mathcal{D}'((\Omega\cap \Sigma \cap U_{\nu})\times \mathbb{S}^{d-1})$ such that, for each $m$, there are $J_{m,\nu}\in \mathbb{N}$ such that
\begin{equation}
\label{TeqSupSigma}
\langle f,\phi \rangle = \sum_{\nu=1}^{\infty} \sum_{j=m}^{J_{m,\nu}} \langle f_{\nu,j}(\xi,\omega), a_{j}^{\phi,\nu}(\xi,\omega) \chi_{\nu}(\xi) \rangle
\end{equation}
for each $\phi \in \mathcal{D}_{\ast,\Sigma}^{[m]} (\Omega)=\bigcup_{K\Subset \Omega}\mathcal{D}_{\ast,\Sigma}^{[m; K]}  (\Omega) .$ Since the restriction of each $\phi\in\mathcal{D}(\Omega)$ to $\Omega\cap U_{\nu}\setminus\Sigma$ has strong asymptotic expansion \eqref{teqExpD}, we obtain that
\[f\in \operatorname*{ker} \Pi \iff \langle f_{\nu,j}(\xi,\omega), \omega^{\alpha} \rangle=0, \quad \forall j\geq 0,\forall \alpha\in\mathbb{N}^{d} \mbox{ with }|\alpha|=j, \]
where $f$ admits the representation \eqref{TeqSupSigma}.
\end{remark}

 \section{Examples of thick distributions along a submanifold}
 \label{Section examples}

We shall now discuss two particular classes of thick distributions along the submanifold $\Sigma$. We fix again an open set $\Omega$ with $\Omega\cap \Sigma \neq\emptyset$. In particular, we shall explain in Subsection \ref{T Fp subsection} how one can embed $\mathcal{E}_{\ast, \Sigma}(\Omega)$ into $\mathcal{D}'_{\ast, \Sigma}(\Omega)$, which, unlike in classical distribution theory, requires a choice of a regularization procedure. In Subsection \ref{T derivatives subsection}, we compute the partial derivatives of various $\Sigma$-thick distributions.
\subsection{Finite part $\Sigma$-thick distributions}
\label{T Fp subsection}
If $\psi\in \mathcal{E}_{\ast, \Sigma}(\Omega)$ and $\phi\in\mathcal{D}_{\ast, \Sigma}(\Omega)$, the integral
\begin{equation}
\label{Teq5.1}
\int_{\Omega} \psi(x)\phi(x)\mathrm{d}x
\end{equation}
will be divergent in general. Observe that this will be the case even if $\psi\in C^{\infty}(\Omega)$. We shall show that the divergent integrals \eqref{Teq5.1} always admit regularization via a special finite part limit procedure. Naturally, the method will work to regularize many other $\psi\in L^{1}_{loc}(\Omega\setminus \Sigma)$.
 
We recall that the finite part \cite{EKBook,EV17} of the limit of $F\left(
\varepsilon\right)  $ as $\varepsilon\rightarrow0^{+}$ exists and equals $\gamma$
if we can write $F\left(  \varepsilon\right)  =F_{\mathrm{fin}}\left(
\varepsilon\right)  +F_{\mathrm{infin}}\left(  \varepsilon\right)  ,$ where
the \emph{infinite} part $F_{\mathrm{infin}}\left(  \varepsilon\right)  $ is a
linear combination of functions of the type $\varepsilon^{-p}\ln
^{q}\varepsilon$ (with $q>0$ if $p=0$) and where the \emph{finite}
part $F_{\mathrm{fin}}\left(  \varepsilon\right)  $ is a function whose limit
as $\varepsilon\rightarrow0^{+}$ is $\gamma$. In such a case, one writes $\gamma=\operatorname*{F.p.} \lim_{\varepsilon \to 0^{+}} F(\varepsilon)$.

\smallskip

\begin{definition}
\label{Definition:FinitePart}Let $f\in L^{1}_{loc}(\Omega\setminus \Sigma)$. The $\Sigma$-thick
distribution $\operatorname*{Pf}\left(  f \right)  $ is
defined as
\begin{equation}
 \label{Dis.3}
\left\langle \operatorname*{Pf}\left(  f\right)  
,\phi  \right\rangle =\mathrm{F.p.}\lim_{\varepsilon\to 0^{+}} \int
_{\Omega\setminus \Sigma^{\varepsilon }}f\left( x\right)  \phi\left(
x\right)  \mathrm{d}x
\end{equation}
provided these finite part limits exist for all $\phi\in\mathcal{D}
_{\ast,\Sigma}\left( \Omega\right)  .$ We also simply denote \eqref{Dis.3} by $\mathrm{F.p.}\int
_{\Omega}f\left( x\right)  \phi\left(
x\right)  \mathrm{d}x$, the (tubular) finite part of the integral along $\Sigma$.
\end{definition}

Let us discuss an important example. We write $\rho$ for the function 
$$
\rho(x)=\operatorname*{dist}(x, \Sigma),
$$
and consider its powers $\rho^{\lambda}$ in the next example. When $x\in\Sigma^{\epsilon}$, we have $\rho(x)=\rho_x=\|x-\pi(x)\|$.

\begin{example}
\label{Example:r a la l}Let $\lambda\in\mathbb{C}$. We verify that
$\operatorname*{Pf}\left(  \rho^{\lambda}\right)  $ is always a well-defined element of $\mathcal{D}_{\ast, \Sigma}^{\prime}\left(  \mathbb{R}^{n}
\right) .$ Let $\phi\in \mathcal{D}_{\ast,\Sigma}(\Omega)$. Using a partition of unity, we may reduce the general case to $\operatorname*{supp}\phi\subset U_{\nu}$. Thus, we may assume the latter, so that the asymptotic expansion \eqref{teq1} holds globally. Then, if $\eta>0$ is such that $\Sigma^{\eta}\cap \operatorname*{sup}\phi \subset \Sigma^{\epsilon}$, we obtain  

\begin{align*}
&\mathrm{F.p.}\int_{\Omega}\rho^\lambda(x)  \phi\left(
x\right)  \mathrm{d}x=\int_{\Omega\setminus\Sigma^{\eta}}\rho^\lambda (x) \phi\left(
x\right)  \mathrm{d}x 
\\
&\qquad  +\int_{\Sigma^{\eta}}\rho_{x}^{\lambda +d-1} \left (\phi(
x)-\sum_{j\leq - \Re e\:\lambda -d }\rho_x^{j} a_{j}(\xi_x,\omega_x)\right)  \mathrm{d}x 
\\&
\qquad
+\sum_{j\leq - \Re e\:\lambda -d }\mathrm{F.p.} \int_{0}^{\eta} \iint_{\Sigma\times \mathbb{S}^{d-1}}
 \rho^{\lambda +j+d-1}a_j(\xi,\omega)   \mathrm{d}\rho\mathrm{d}\sigma(\xi,\omega),
\end{align*}
so that

\begin{align*}
\langle \operatorname*{Pf}(r^{\lambda}), \phi\rangle=&\int_{\Omega\setminus\Sigma^{\eta}}\rho^\lambda (x) \phi\left(
x\right)  \mathrm{d}x +\int_{\Sigma^{\eta}}\rho_{x}^{\lambda +d-1} \left (\phi(
x)-\sum_{j\leq - \Re e\:\lambda -d }\rho_x^{j} a_{j}(\xi_x,\omega_x)\right)  \mathrm{d}x 
\\& \qquad
+ \sum_{j\leq - \Re e\:\lambda -d }\frac{\eta^{\lambda+j+d}}{\lambda+j+d} \iint_{\Sigma\times \mathbb{S}^{d-1}}
a_j(\xi,\omega)\mathrm{d}\sigma(\xi,\omega) 
\end{align*}
if $\lambda\notin \mathbb{Z}$, while if $\lambda=k \in\mathbb{Z}$, we have
\begin{align*}
&\langle \operatorname*{Pf}(r^{k}), \phi\rangle=\int_{\Omega\setminus\Sigma^{\eta}}\rho^k (x) \phi\left(
x\right)  \mathrm{d}x +\int_{\Sigma^{\eta}}\rho_{x}^{k +d-1} \left (\phi(
x)-\sum_{j\leq - k -d }\rho_x^{j} a_{j}(\xi_x,\omega_x)\right)  \mathrm{d}x 
\\&
+ \ln \eta \iint_{\Sigma\times \mathbb{S}^{d-1}}
a_{-k-d}(\xi,\omega)\mathrm{d}\sigma(\xi,\omega) +\sum_{j< - k -d } \frac{\eta^{k+j+d}}{k+j+d}\iint_{\Sigma\times \mathbb{S}^{d-1}}
a_j(\xi,\omega)\mathrm{d}\sigma(\xi,\omega).
\end{align*}

\end{example}

\bigskip

Example \ref{Example:r a la l}  then yields the following corollary.

\begin{corollary}
\label{Tcor1} $\operatorname*{Pf}(\psi)\in \mathcal{D}'_{\ast,\Sigma}(\Omega)$ exists for all $\psi \in\mathcal{E}_{\ast,\Sigma}(\Omega)$.
\end{corollary}
\begin{proof} In fact, $\operatorname*{Pf}(\psi)= \psi \cdot \operatorname*{Pf}(1)$.
\end{proof}

We therefore have a way to embed  $\mathcal{E}_{\ast,\Sigma}(\Omega)$ into $\mathcal{D}'_{\ast,\Sigma}(\Omega)$ via the continuous linear mapping $ \operatorname*{Pf}: \mathcal{E}_{\ast,\Sigma}(\Omega)\to\mathcal{D}'_{\ast,\Sigma}(\Omega)$. On the other hand, we warn the reader that one should proceed with care at manipulating this embedding, because for example, it does not commute with partial derivatives (see Corollary \ref{cor der test} in Subsection \ref{T derivatives subsection} below).

\subsection{Thick deltas and multilayers along $\Sigma$}
\label{T deltas subsection} 
In analogy to the Dirac delta, it is useful to consider the ``delta function'' concentrated at $\Sigma$, that is, the distribution $\delta_{\Sigma}\in\mathcal{D}'(\mathbb{R}^{n})$ given by 
$$
\langle \delta_{\Sigma}, \varphi \rangle=\int_{\Sigma} \varphi(\xi)\mathrm{d}\sigma(\xi), \qquad \varphi\in \mathcal{D}(\mathbb{R}^{n}).
$$
More generally \cite{EKBook,K04}, if $h\in \mathcal{D}'(\Sigma)$, its associated single layer distribution $h\delta_{\Sigma}\in\mathcal{D}'(\mathbb{R}^{n})$ is defined by
$$
\langle h \delta_{\Sigma}, \varphi \rangle=\langle h, \left.\varphi\right|_{\Sigma}\rangle \qquad \varphi\in \mathcal{D}(\mathbb{R}^{n}).
$$

We wish to consider 
$\Sigma$-thick generalizations of these distributions. For the sake of exposition\footnote{It is still possible to drop the orientability assumption in this section by using partitions of unity and working locally, but we have decided not to do so in order to keep the exposition more transparent.}, we assume in the remainder of this subsection that the manifold is orientable, so that it is possible to choose a global normal orthonormal smooth frame $\{\mathbf{n}_{1}, \dots, \mathbf{n}_{d}\}$, which we now fix, and then the defining asymptotic expansions \eqref{teq1}  for test functions $\phi\in\mathcal{D}_{\ast,\Sigma}(\mathbb{R}^{n})$ hold globally on the $\epsilon$-neighborhood $\Sigma^{\epsilon}$.

Let $|\mathbb{S}^{d-1}|=\int_{\mathbb{S}^{d-1}} \mathrm{d}\omega$. We start by introducing the ``plain'' $\Sigma$-think delta distribution $\delta_{\ast,\Sigma}\in \mathcal{D}'_{\ast, \Sigma}(\mathbb{R}^{n})$ via
\[
\langle \delta_{\ast,\Sigma}, \phi \rangle=\frac{1}{|\mathbb{S}^{d-1}|}\iint_{\Sigma\times \mathbb{S}^{d-1}} a_{0}(\xi,\omega)\mathrm{d}\sigma(\xi)\mathrm{d}\sigma(\omega), \qquad \phi\in \mathcal{D}_{\ast, \Sigma}(\mathbb{R}^{n}).
\]
When $\phi\in\mathcal{D}(\mathbb{R}^{n})$, $a_{0}(\xi,\omega)=\phi(\xi)$, whence $\Pi(\delta_{\ast,\Sigma})=\delta_{\Sigma}$. 

If $g\in \mathcal{D}'((\Omega\cap \Sigma)\times \mathbb{S}^{d-1})$, we can also introduce the $\Sigma$-thick single layer distribution $g \delta_{\ast,\Sigma}\in \mathcal{D}'_{\ast, \Sigma}(\Omega)$ as
\[
\langle g\delta_{\ast,\Sigma}, \phi \rangle=\frac{1}{|\mathbb{S}^{d-1}|}\langle g, a_{0}\rangle , \qquad \phi\in \mathcal{D}_{\ast, \Sigma}(\Omega).
\]
We have that $\Pi(g\delta_{\ast,\Sigma})= h \delta_{\Sigma}$ is a single layer along $\Sigma\cap\Omega$, where $h\in\mathcal{D}'(\Sigma)$ is given by $h(\xi)= |\mathbb{S}^{d-1}|^{-1}\langle g(\xi,\omega), 1 \rangle_{\mathcal{D}'(\mathbb{S}^{d-1})\times \mathcal{D}(\mathbb{S}^{d-1})}$, i.e., 
\[
\langle h\delta_{\Sigma}, \varphi \rangle=\frac{1}{|\mathbb{S}^{d-1}|}\langle g, \left.\varphi\right|_{\Sigma}\otimes 1\rangle , \qquad \varphi\in \mathcal{D}(\Omega).
\]
In particular, when $g$ just depends on $\omega$, that is, if it has the form $g=1 \otimes g_0$ with $g_0\in \mathcal{D}'(\mathbb{S}^{d-1})$, we obtain that $\Pi(g\delta_{\ast,\Sigma})=c\delta_{\Sigma}$, where the constant $c$ is 
$$
c= \frac{1}{|\mathbb{S}^{d-1}|}\langle g_0, 1\rangle_{\mathcal{D}'(\mathbb{S}^{d-1})\times \mathcal{D}(\mathbb{S}^{d-1})}.
$$

We can generalize these ideas to $\Sigma$-thick multilayer distributions:

\begin{definition}
Let $g \in\mathcal{D}'((\Omega\cap \Sigma)\times \mathbb{S}^{d-1}) $ and $j\in\mathbb{Z}$. The $\Sigma$-thick delta function of degree $j$ associated with $g$ is the  $\Sigma$-thick distribution $g\delta_{\ast,\Sigma}^{[ j] } \in \mathcal{D}'_{\ast, \Sigma}(\Omega)$ whose action on $\phi\in\mathcal{D}_{\ast,\Sigma}(\Omega)$ is given by
\begin{equation*}
\left\langle g\delta_{\ast,\Sigma}^{[ j]  },\phi\right\rangle
_{\mathcal{D}_{\ast,\Sigma}^{\prime}\left( \Omega \right)
\times\mathcal{D}_{\ast, \Sigma}\left(  \Omega \right)  }=\frac
{1}{|\mathbb{S}^{d-1}|}\left\langle g (\xi,\omega)  ,a_j(\xi,\omega)  \right\rangle.
\end{equation*}
If $j=0$, we simply write $g\delta_{\ast,\Sigma}^{\left[ 0\right] } =g\delta_{\ast,\Sigma} $, as defined before.
\end{definition}

\smallskip

Interestingly, Remark \ref{trmproj} can be rephrased in terms of $\Sigma$-thick deltas. We should call a two-sided sequence $\{g_j\}_{j=-\infty}^{\infty}$ of distributions locally upper finite if on any relatively compact open set $W$ there is $J$ such that ${g_{j}}_{|W}=0$ for all $j>J$. 
\begin{proposition} Let $f\in \mathcal{D}'_{\ast,\Sigma}(\Omega)$. Then $\operatorname*{supp}f\subseteq \Sigma$ if and only if there is an upper locally finite two-sided sequence of distributions $\{g_j\}_{j=-\infty}^{\infty}$ of $\mathcal{D}'(\Sigma\times \mathbb{S}^{d-1})$ such that $f=\sum_{j\in\mathbb{Z}} g_{j}\delta^{[j]}_{\ast,\Sigma}$. Furthermore, $f\in\operatorname*{ker}\Pi$ if and only if $\langle g_{j}(\xi,\omega), \omega^{\alpha}  \rangle=0$,  $ \forall j\geq 0,$$\forall \alpha\in\mathbb{N}^{d}$ with $|\alpha|=j$.
\end{proposition}

Observe that $
\Pi(  g\delta_{\ast,\Sigma}^{\left[  j\right]  })
=0$  whenever $j<0$. We already treated the case where $j=0$. The projections of the $\Sigma$-thick deltas for $j>0$ are calculated in the following proposition, which follows at once from \eqref{Teq3.3}, but we first need to introduce some useful notation in order to state our result. Let $\varphi\in C^{\infty}(\Omega)$ and set $\Psi(\xi,y)=\xi + \sum_{k=1}^{d} y_k \mathbf{n}_{k}^{\nu}(\xi)$. If $\alpha \in \mathbb{N}^{d}$, we define the derivative $D_{\mathbf{n}}^{\alpha}\varphi \in C^{\infty}(\Sigma\cap \Omega)$ with respect to the normal frame as
\[
D_{\mathbf{n}}^{\alpha}\varphi= \left.\frac{\partial^{\alpha}(\varphi\circ \Psi_{\nu})}{\partial y^{\alpha}} \right|_{(\Sigma\cap \Omega) \times \{0\}}\]
and extend it to single layers $h\delta_{\Sigma}\in \mathcal{D}'(\Omega)$ as
$$
\langle \bar{D}_{\mathbf{n}}^{\alpha}(h\delta{\Sigma}), \varphi \rangle = (-1)^{|\alpha|} \langle h, D_{\mathbf{n}}^{\alpha}\varphi \rangle .
$$
\begin{remark}
\label{rm coefficients b} With this notation, the functions $b_{l,i,q}$ from \eqref{eqdef b coefficient} become
\begin{equation}
\label{other formula b}
b_{l,i,q}(\xi,\omega)= \sum_{|\alpha|=q}\frac{\omega^{\alpha}}{\alpha !} D_{\mathbf{n}}^{\alpha}\left(\frac{\partial \xi_{l}}{\partial x_i}\right) .
\end{equation}
\end{remark}
\begin{proposition}
\label{Proposition Projection}If  $g\in\mathcal{D}'(
(\Sigma\cap \Omega)\times \mathbb{S}^{d-1}) $ and $j\geq0$, then
\begin{equation*}
\Pi (g\delta_{\ast,\Sigma}^{[ j]})
=\frac{\left(
-1\right)  ^{j}}
{|\mathbb{S}^{d-1}|}
\sum_{|\alpha|=j}\frac{\bar{D}_{\mathbf{n}}^{\alpha}(h_{\alpha}\delta
_{\Sigma})}{\alpha!},
\end{equation*}
where $h_{\alpha}(\xi)=\langle g(\xi,\omega),\omega^{\alpha} \rangle\in \mathcal{D}'(\Sigma)$.
\end{proposition}

Proposition \ref{Proposition Projection} tells us that, using the terminology from \cite{EKBook,K04}, the $\Sigma$-thick delta functions are generalizations of the so-called multilayers on $\Sigma$.

We end this subsection mentioning an important relation between $\operatorname*{Pf}\left(  \rho^{\lambda}\right)  $ and the $\Sigma$-thick delta distributions. It directly follows from the formulas given in Example \ref{Example:r a la l}.

\begin{proposition}
The  $\Sigma$-thick distributions $\operatorname*{Pf}( \rho^{\lambda})  $ form an 
analytic family with respect to the parameter $\lambda$ in the region $\mathbb{C}\setminus\mathbb{Z}.$ Moreover,
$\operatorname*{Pf}( \rho^{\lambda})  $ has simple poles at all integers $k\in\mathbb{Z}$ with
residues
\begin{equation*}
\operatorname*{Res}_{\lambda=k}\operatorname*{Pf}( \rho^{\lambda})  =|\mathbb{S}^{d-1}| \ \delta_{\ast,\Sigma}^{[ -k-d] }.
\end{equation*}
In addition, the  $\Sigma$-thick distributions $\operatorname*{Pf}\left(  \rho^{k}\right)  $ satisfy 
\[
\operatorname*{Pf}\left(  \rho^{k}\right)  =\lim_{\lambda\rightarrow
\mathsf{k}}\left( \operatorname*{Pf}\left(  \rho^{\lambda}\right)  -\frac
{|\mathbb{S}^{d-1}| \ \delta_{\ast,\Sigma}^{[ -k-d]  }}
{\lambda-k}\right)  .
\]
\end{proposition}

\subsection{Partial derivatives of some $\Sigma$-thick distributions}
\label{T derivatives subsection} We again assume that $\Sigma$ is orientable. We compute here the partial derivatives of some thick  distributions along $\Sigma$. 

We noticed in the proof of Proposition \ref{tp1} that the partial derivatives of $\rho$ for $x \in \Sigma^{\epsilon}$ do not depend on $\rho_x=\|x-\pi(x)\|$ itself, but only on $(\xi_{x},\omega_{x})$. In fact, the following lemma was already shown there. We write $\mathbf{n}_k=(n_{k,1},...,n_{k,n})$.

	\begin{lemma}\label{lcderivativesrho}
		For $x \in \Sigma^{\epsilon}$, 
		\begin{equation}\label{Equation1}
			\frac{\partial \rho}{\partial x_i}= \sum\limits_{k=1}^{d}\omega_{x,k} \cdot n_{k,i}(\pi(x)).
		\end{equation}
	\end{lemma}

	We define
	\begin{equation}
 \label{thetaieq}
		\vartheta _i (\xi,\omega)=\sum_{k=1}^d \omega_{k}n_{k,i}(\xi), \qquad (\xi, \omega)\in \Sigma \times \mathbb{S}^{d-1},
	\end{equation}
	so that  \eqref{Equation1} reads $\partial \rho/\partial x_i=\vartheta _i (\xi_{x},\omega_{x}),$ for any $ x \in \Sigma ^{\epsilon}$.
	We have $\partial \rho/\partial x_i \in \mathcal{E}_{\ast,\Sigma}(\mathbb{R}^n)$ is a smooth function on $\mathbb{R}^n \setminus \Sigma$, while $\vartheta_i \in C^{\infty} (\Sigma \times \mathbb{S}^{d-1})$.
	
	\begin{theorem} \label{theo}
		The partial derivatives of $\operatorname*{Pf}(\rho^{\lambda})$ are given by 
		\begin{equation*}
			\frac{\partial_* \operatorname*{Pf}(\rho^{\lambda})}{\partial x_i}=\lambda \frac{\partial \rho}{\partial x_i} \operatorname*{Pf}(\rho^{\lambda-1}) \qquad \text{  if  } \lambda \notin \mathbb{Z},
		\end{equation*}
		and 
		\begin{equation*}
			\frac{\partial_{*} \operatorname*{Pf}(\rho^{k})}{\partial x_i}=k \frac{\partial \rho}{\partial x_i} \operatorname*{Pf}(\rho^{k-1}) + |\mathbb{S}^{d-1} |\vartheta _i \delta_{\ast,\Sigma}^{[1-d-k]}  \qquad \text{  if  } k \in \mathbb{Z}.
		\end{equation*}
	\end{theorem}
	\begin{proof}
		Let $\phi \in \mathcal{D}_{*, \Sigma}(\Omega)$ be such that $\phi$ vanishes on $\Omega \setminus K $ with $ K $ compact. Then there is $\varepsilon$ such that $ K \cap \Sigma ^{\varepsilon} \subseteq K \cap \Sigma ^{\epsilon}$. We point out that $\nabla \rho$ is the outer unit normal of hypersurface $\partial\Sigma ^{\varepsilon} =\left\{y \in \Omega : \rho(y)=\varepsilon\right\}$. Therefore, Stoke's theorem yields
			\begin{align*}
				-\left \langle \operatorname*{Pf}(\rho^{\lambda}), \frac{\partial \phi}{\partial x_i} \right \rangle
				&=-\mathrm{F.p.} \lim_{\varepsilon \rightarrow 0^+} \int_{\mathbb{R}^n \setminus \Sigma ^{\varepsilon}} \rho^{\lambda} (x) \frac{\partial \phi}{\partial x_i} \mathrm{d}x\\ 
				&= \lambda \mathrm{F.p.} \int_{\mathbb{R}^n } \rho^{\lambda -1} (x) \frac{\partial \rho}{\partial x_i} \phi(x) \mathrm{d}x 
    \\
    &
    \qquad \qquad \qquad  +\mathrm{F.p.}\lim_{\varepsilon \rightarrow 0^+} \int_{\partial \Sigma ^{\varepsilon}} \rho^{\lambda} (y) \left.\frac{\partial \rho}{\partial x_i}\right|_{y} \phi(y) \mathrm{d} \sigma_{\partial \Sigma ^{\varepsilon}}(y)\\
				&=\left \langle \lambda \frac{\partial \rho}{\partial x_i} \operatorname*{Pf}(\rho ^{\lambda -1}), \phi \right \rangle + \mathrm{F.p.}\lim_{\varepsilon \rightarrow 0^+} \varepsilon^{\lambda}\int_{\partial \Sigma ^{\varepsilon}} \left.\frac{\partial \rho}{\partial x_i}\right|_{y} \phi(y) \mathrm{d} \sigma_{\partial \Sigma ^{\varepsilon}}(y).
			\end{align*}
		Now, if $\phi(x) \sim \sum_{j=m}^{\infty}a_j(\xi_{x}, \omega_{x})(\rho(x))^{j}$ as $\rho(x) \rightarrow 0$, we have
		\[
			\varepsilon ^{\lambda} \int_{\partial \Sigma ^{\varepsilon}} \left.\frac{\partial \rho}{\partial x_i}\right|_{y} \phi(y) \mathrm{d} \sigma_{\partial \Sigma ^{\varepsilon}}(y) \sim \sum_{j=m}^{\infty} \varepsilon ^{\lambda +j+d-1} \int \int_{\Sigma \times \mathbb{S}^{d-1}} \vartheta _i(\xi, \omega)a_j(\xi, \omega) \mathrm{d} \sigma_{\Sigma}(\xi) \mathrm{d}\sigma_{\mathbb{S}^{n-1}}(\omega),
		\]
		and hence, we obtain the result.
	\end{proof}
	As an application of Theorem \ref{theo}, we can compute the partial derivatives of $\operatorname*{Pf}(\psi)$ where $\psi \in \mathcal{E}_{*,\Sigma}(\Omega)$.
	\begin{corollary}\label{cor der test}
		If $\psi \in \mathcal{E}_{*,\Sigma}(\Omega)$, then
		\begin{equation}\nonumber
			\frac{\partial_{*} \left(\operatorname*{Pf}(\psi)\right)}{\partial x_i}=\operatorname*{Pf}\left(\frac{\partial \psi}{\partial x_i}\right) + |\mathbb{S}^{d-1}|\psi \cdot \vartheta_i \delta_{\ast,\Sigma}^{[1-d]}.
		\end{equation}
	\end{corollary}
	\begin{proof}
		By Proposition \ref{td proposition differentiation} and Theorem \ref{theo},
			\begin{align*}
				\frac{\partial_{*} (\operatorname*{Pf}(\psi))}{\partial x_i}&=\frac{\partial_{*}}{\partial x_i} (\psi \operatorname*{Pf}(1))=\frac{\partial \psi}{\partial x_i} \operatorname*{Pf}(1)+|\mathbb{S}^{d-1}|\psi \cdot \vartheta_i \delta_{\ast,\Sigma}^{[1-d]}\\
				&=\operatorname*{Pf}\left(\frac{\partial \psi}{\partial x_i}\right)+|\mathbb{S}^{d-1}|\psi \cdot \vartheta_i \delta_{\ast,\Sigma}^{[1-d]}.
			\end{align*}
		
	\end{proof}

The partial derivatives of thick deltas and, more generally, of multilayers along $\Sigma$ can directly be computed from Corollary \ref{tcfpartialdefivatives} and Lemma \ref{lcderivativesrho}. We only state this in the case of first order partial derivatives, but of course higher order ones could be found by repeated application of Proposition \ref{pcalderdelta}.

\begin{proposition}
\label{pcalderdelta}
Let $g \in\mathcal{D}'((\Omega\cap \Sigma)\times \mathbb{S}^{d-1}) $ and $j\in\mathbb{Z}$. The partial derivative of the  $\Sigma$-thick distribution $g\delta_{\ast,\Sigma}^{[ j] } \in \mathcal{D}'_{\ast, \Sigma}(\Omega)$ with respect to the variable $x_i$ is
\begin{equation*}
\left\langle \frac{\partial} {\partial x_i}\left(g\delta_{\ast,\Sigma}^{[ j]  }\right),\phi\right\rangle
=- \frac
{1}{|\mathbb{S}^{d-1}|}\left\langle g (\xi,\omega)  ,a_{j,i}(\xi,\omega)  \right\rangle,
\end{equation*}
where and smooth function $a_{i,j}$ on $\Sigma\times \mathbb{S}^{d-1}$ is explicitly given by
\begin{align}\label{eqderfordelta}
a_{j,i}&= (j+1) a_{j+1}\vartheta_{i} +  \frac{\delta a_{j}}{\delta \xi_{i} } + \sum_{k=m}^{j-1} \sum_{l=1}^{n} b_{l,i,j-k} \frac{\delta a_k}{\delta \xi_l}
\\
& \notag
\qquad \qquad +\sum_{k=1}^{d}  \left(\frac{\delta a_{j+1}}{\delta \omega_{k} }(n_{k,i}- {\omega_{k}}\vartheta_{i})
+\frac{\delta a_{j}}{\delta \omega_{k} }\sum_{l=1}^{d}\omega_{l} \mathbf{n}_{l} \cdot \frac{\delta \mathbf{n}_{k}}{\delta x_{i}}
\right)  \: ,
\end{align}
where the functions $b_{l,i,q}$ are given by \eqref{other formula b} and $\phi$ has asymptotic expansion \eqref{teq1}.

\end{proposition}
\begin{proof} Insert \eqref{thetaieq} in the formula from Corollary \ref{tcfpartialdefivatives}.
\end{proof}
\begin{example}
In the case of a hypersurface $\Sigma$, i.e., when $d=1$, the expression \eqref{eqderfordelta} simplifies because the $\delta$-derivatives with respect to $\omega_{i}$ do not occur in this case. We can write $\phi(x)=\phi^{+}(x)+\phi^{-}(x)$ where $\phi^{\pm}$ vanishes according to whether $x$ lies ``inside'' or ``outside'' $\Sigma$. We also have $\omega_{x}=1$ if $x$ lies in the region toward which the unit normal $\mathbf{n}$ points, and $\omega_{x}=-1$ otherwise. An arbitrary thick multilayer $g\delta^{[j]}_{\ast,\Sigma}$ on the hypersurface has the form
\[
\langle g\delta^{[j]}_{\ast,\Sigma},\phi \rangle= \frac{1}{2} \left( \langle g^{+}(\xi),a^{+}_{j}(\xi) \rangle + \langle g^{-}(\xi),a^{-}_{j}(\xi) \rangle \right),
\]
where $g^{\pm}\in\mathcal{D}'(\Sigma)$ and $\phi^{\pm}(x) \sim \sum_{j=m}^{\infty}a_j^{\pm}(\xi_{x})\rho_{x}^{j}$ as $\rho_{x} \rightarrow 0$. Since there is no longer any risk of confusion, we can simply write $\delta{a^{\pm}_{j}}/\delta x_i= \delta{a^{\pm}_{j}}/\delta {\xi_i}$. We also have $\vartheta=\pm1$ according to the considered case.
Therefore, we have the following general formula for the partial derivatives of a thick multilayer on the hypersurface $\Sigma$,
\begin{align*}
\left\langle \frac{\partial} {\partial x_i}\left(g\delta_{\ast,\Sigma}^{[ j]  }\right),\phi\right\rangle
=&
 \frac{j+1}{2} \:\left(\left\langle g ^{-}, a^{-}_{j+1}n_{i}  \right\rangle - \left\langle g ^{+}, a^{+}_{j+1}n_{i}  \right\rangle\right)
 \\
&
 -
\frac
{1}{2}\Big(\big\langle g ^{+}, \frac{\delta a^{+}_{j}}{\delta x_{i} }\big\rangle+\big\langle g ^{-}, \frac{\delta a^{-}_{j}}{\delta x_{i} }\big\rangle\Big)
 \\
&
 -
\frac
{1}{2}\sum_{k=m}^{j-1} \sum_{l=1}^{n}\Big(\big\langle g ^{+}, b_{l,i,j-k}^{+} \frac{\delta a^{+}_k}{\delta x_l}\big\rangle+\big\langle g ^{-},  b_{l,i,j-k}^{-} \frac{\delta a^{-}_k}{\delta x_l}\big\rangle\Big).
\end{align*}
In particular, we obtain
\begin{align}
\label{eqforexamples}
\left\langle \frac{\partial \delta_{\ast,\Sigma}^{[ j] }} {\partial x_i},\phi\right\rangle
= & \frac{1}{2}\int_{\Sigma}\left(-\frac{\delta a^{+}_{j}}{\delta x_{i} }-\frac{\delta a^{-}_{j}}{\delta x_{i} }+(j+1) (a_{j+1}^{-}(\xi)-a^{+}_{j+1}(\xi))n_{i}(\xi) \right.
\\
& \notag
\qquad \qquad
- \left. \sum_{k=m}^{j-1} \sum_{l=1}^{n}\left(b_{l,i,j-k}^{+}(\xi)\frac{\delta a^{+}_k}{\delta x_l}+ b_{l,i,j-k}^{-}(\xi)\frac{\delta a^{-}_k}{\delta x_l} \right)  \right) \mathrm{d}\sigma(\xi).
\end{align}
\end{example}

\begin{example}
Let $\Sigma\subset \mathbb{R}^{n}$ be a compact hypersurface with normal unit vector $\mathbf{n}$. Let
\[\mu_{ik}=\frac{\delta n_{k}}{\delta x_{i}},\]
where as usual $n_k$ are the coordinate functions of $\mathbf{n}$.  The functions $\mu_{ij}$ are the matrix coefficients of the second fundamental form of $\Sigma$ expressed in the Cartesian coordinates of the surrounding space $\mathbb{R}^{n}$. The mean curvature of the hypersurface can then be written as
\[
H=  \frac{1}{n-1}\sum_{i=1}^{n} \mu_{ii}.
\]
If we select $\phi_{k}\in\mathcal{E}_{\ast,\Sigma}(\mathbb{R}^{n})$ having asymptotic expansion $\phi^{\pm}_{k}(x)\sim n_{k}(\xi_x)$ as $\rho_{x}\to 0$, then 
\eqref{eqforexamples} yields
\begin{equation}
\label{formulaIIfundamentalformintegrals}
\left\langle \frac{\partial \delta_{\ast,\Sigma}^{[ j] }} {\partial x_i},\phi_{k}\right\rangle
= 
\begin{cases}
\displaystyle -\int_{\Sigma} \mu_{ik}(\xi) \mathrm{d}\sigma(\xi) &\mbox{ if } j=0,\\
0  &\mbox{ if } j\leq -1,
\\
\displaystyle - \frac{1}{2} \sum_{l=1}^{n}\int_{\Sigma} (b_{l,i,j}^{+}(\xi)+b^{-}_{l,i,j}(\xi))\mu_{l,k}(\xi) \mathrm{d}\sigma(\xi) &\mbox{ if } j\geq 1.
\end{cases}
\end{equation}
We thus get
\[
\sum_{i=1}^{n}\left\langle \frac{\partial \delta_{\ast,\Sigma}}{\partial x_i},\phi_{i}\right\rangle
=(1-n)\int_{\Sigma} H(\xi)\mathrm{d}\sigma(\xi),
\]
and, in particular,
\[
\int_{\Sigma} H(\xi)\mathrm{d}\sigma(\xi)=\frac{1}{1-n} \left\langle\nabla \delta_{\ast,\Sigma} \, , \mathbf{n}\right\rangle ,
\]
if we interpret $\mathbf{n}$ as the vector field $\mathbf{n}(\xi_x)$ defined on a neighborhood of $\Sigma$.
\end{example}

\begin{example}
Let us conclude the article with a hands-on simple example where we compute \eqref{formulaIIfundamentalformintegrals}. We consider the sphere $\Sigma= r\mathbb{S}^{n-1}$ in $\mathbb{R}^{n}$. Naturally, we only need to treat $j\geq0$. The normal unit vector at  $\xi\in r\mathbb{S}^{n-1}$ is simply 
$\mathbf{n}(\xi)= \xi/r= (\xi_1/r, \dots, \xi_{n}/r)$. If $x\neq0$, then $\xi_{x}=\pi(x)=rx/\|x\|$. Hence,
\[
\mu_{ik}= \frac{\delta n_{k}}{\delta x_i}= \frac{1}{r}\left(\delta_{ik}- \frac{\xi_i \xi_k}{r^2}\right)=\frac{1}{r}\left(\delta_{ik}- n_i n_k\right),
\]
and so
\[
\left\langle \frac{\partial \delta_{\ast,r\mathbb{S}^{n-1}}} {\partial x_i},\phi_{k}\right\rangle=-r^{n-2}\left(|\mathbb{S}^{n-1}|  \delta_{ik} - \int_{\mathbb{S}^{n-1}}\xi_i\xi_k \mathrm{d}\sigma(\xi)\right)=\delta_{ik} r^{n-2} |\mathbb{S}^{n-1}|\left( \frac{1}{n}-1\right).
\]

For $j\geq 1$, we need to find the coefficients $b^{\pm}_{l,i.j}$. We use \eqref{eqdef b coefficient} for it. We have that
\[ \frac{\partial \xi_{l}}{\partial x_i}= r\left(\frac{\delta_{il}}{\|x\|}- \frac{x_i x_l}{\|x\|^{3}}\right)
= \frac{r}{\|x\|}\left(\delta_{il} - \frac{\xi_i \xi_l}{r^{2}}\right).
\]
Since $\rho=\rho_x= |\|x\|-r|$, 

\[
\left(\frac{\partial \xi_{l}}{\partial x_i}\right)^{\pm}= \frac{\delta_{il} - \xi_i \xi_l/r^2}{1\pm \rho/r} .
\]
Expanding into geometric series, we obtain
\[
b^{\pm}_{l,i,j}= \left(\delta_{il}-\frac{\xi_i \xi_l}{r^2} \right) \left(\mp \frac{1} {r}\right)^{j} .
\]
If $j$ is odd, we directly obtain 
\[
\left\langle \frac{\partial \delta_{\ast,r\mathbb{S}^{n-1}}^{[ j] }} {\partial x_i},\phi_{k}\right\rangle=0.\]
When $j$ is even

\begin{align*}
\left\langle \frac{\partial \delta_{\ast,r\mathbb{S}^{n-1}}^{[ j] }} {\partial x_i},\phi_{k}\right\rangle& =- \frac{1}{r^{j+1}}\sum_{l=1}^{n} \int_{r\mathbb{S}^{n-1}} \left(\delta_{il}-\frac{\xi_i\xi_l}{r^2}\right)\left(\delta_{lk}-\frac{\xi_l\xi_k}{r^2}\right) \mathrm{d}\sigma(\xi)
\\
&= \delta_{ik} r^{n-j-2} |\mathbb{S}^{n-1}|\left( \frac{1}{n}-1\right)+ r^{n-j-2} \int_{\mathbb{S}^{n-1}} \xi_i \xi_k\mathrm{d}\sigma(\xi)
\\
& \qquad \qquad -  r^{n-j-2}\sum_{l=1}^{n} \int_{\mathbb{S}^{n-1}} \xi_i \xi_l^2 \xi_k \mathrm{d}\sigma(\xi)
\\
&
=\delta_{ik} r^{n-j-2} |\mathbb{S}^{n-1}|\left( \frac{1}{n}-1\right).
\end{align*}
\end{example}	

\section{Conclusion}
This article generalizes the theory of thick distributions from \cite{YE2013,Yang2022} by considering general ``singular submanifolds'' instead of just singular curves or singular points inside the domains of the test functions. Given a closed submanifold $\Sigma$ of $\mathbb{R}^{n}$, we have introduced the test function space  $\mathcal{D}_{\ast, \Sigma}(\Omega)$, which consists of smooth functions on $\Omega\setminus\Sigma$ admitting asymptotic expansions of the form \eqref{teq1}, where $(\xi,\omega,\rho)$ are tubular coordinates on an $\epsilon$-neighborhood of $\Sigma$. This test function space is a Montel space and is invariant under the action of partial differential operators. We have computed asymptotic expansions for the partial derivatives of the elements of $\mathcal{D}_{\ast, \Sigma}(\Omega)$ when approaching the singular submanifold. 

The dual space $\mathcal{D}'_{\ast, \Sigma}(\Omega)$ is the space of thick distributions along $\Sigma$. It carries a natural action of the partial derivative operators and its Moyal algebra is the space $\mathcal{E}_{\ast, \Sigma}(\Omega)$. There is a natural projection from $\mathcal{D}'_{\ast, \Sigma}(\Omega)$ onto the space of Schwartz distributions on $\Omega$ which commutes with partial derivatives. The kernel of this projection was determined. We introduced regularizations of functions via tubular finite-part integrals and constructed an embedding of $\mathcal{E}_{\ast, \Sigma}(\Omega)$ into $\mathcal{D}'_{\ast, \Sigma}(\Omega)$. This embedding does not commute with partial derivatives. We also defined new thick delta distributions along $\Sigma$ and, more generally, thick multilayers. Their projections onto the Schwartz distribution space and their partial derivatives were found. 

The theory of thick distributions might have potential applications in physics and PDE theory. As indicated in the Introduction, several applications of the theory of distributions with thick points and thick curves have been devised in recent years. The new theory presented in this article provides the possibility to treat thick submanifolds.

\end{document}